\documentclass[12pt,a4paper]{amsart}

\usepackage{amsmath}
\usepackage{amsfonts}
\usepackage{amssymb}
\usepackage{amsthm}
\usepackage[a4paper,top=2cm,bottom=2cm,left=2.5cm,right=2.5cm,marginparwidth=1.75cm]{geometry}
\usepackage{setspace}
\usepackage[T2A]{fontenc}% unicode characters
\usepackage{enumerate} % http://ctan.org/pkg/enumerate
\newtheorem{theorem}{Theorem}[section]
\newtheorem{lemma}[theorem]{Lemma}

\newtheorem{remark}[theorem]{Remark}

\newtheorem{corollary}[theorem]{Corollary}
\newtheorem{proposition}[theorem]{Proposition}
\numberwithin{equation}{section}

\title{Accretive Partial Transpose Matrices and Their Connections to Matrix Means}
\author{Eman Aldabbas}
\address{Department of Mathematics, University of Jordan\\ Amman 11942, Jordan}
\email{e\_aldabbas@ju.edu.jo,aldabbas@ualberta.ca}

\author{Mohammad Sababheh}
\address{Department of Basic Sciences, Princess Sumaya University for Technology\\ Amman 11941, Jordan}
\email{sababheh@psut.edu.jo; sababheh@yahoo.com}

\begin{document}

\date{}


\begin{abstract}
    
Accretive partial transpose (APT) matrices have been recently defined, as a natural extension of positive partial transpose (PPT) matrices.

In this paper, we discuss further properties of APT matrices in a way that extends some of those properties known for PPT matrices. 

Among many results, we show that
 if \(A,B,X\) are $n\times n$ complex matrices such that \(A,B\) are sectorial with sector angle $\alpha$ for some \(\alpha\in [0,\pi/2)\), and if \(f:(0,\infty)\to(0,\infty)\) is a certain operator monotone function such that \(\begin{bmatrix}
        \cos^2(\alpha) f(A)  & X\\
        X^* & \cos^2(\alpha) f(B)
    \end{bmatrix}\) is APT, Then \(\begin{bmatrix}
        f(A)\nabla_t f(B) & X\\
        X^* & f(A \nabla_tB )
    \end{bmatrix}\) is APT for any \(0\leq t\leq 1\), where $\nabla_t$ is the weighted arithmetic mean.
\end{abstract}

\maketitle
\section{Introduction}\label{sec1}

In the sequel, upper case letters will be used to denote square matrices of appropriate sizes. The zero matrix will be denoted by $O$, while the identity matrix by $I$. The algebra of all $n\times n$ matrices is denoted by $\mathcal{M}_n$. Thus, if $A,B,C,D\in\mathcal{M}_n$, then the $2\times 2$ block form $\left[\begin{matrix}A&B\\C&D\end{matrix}\right]$ is an element of $\mathcal{M}_{2n}.$

Block matrices have acquired a considerable attention in the literature, as they  can be used to better understand certain notions and to simplify some proofs. We refer the reader to \cite{gumus2022positive,huang2022refining, lee2015off,li2021partial,sababheh2023matrix,sababheh2023new,southworth2020note} where applications of block matrices can be seen.

An important subclass of $\mathcal{M}_n$ is the cone of positive matrices. Recall that $A\in\mathcal{M}_n$ is said to be positive semi-definite, and written as $A\geq O$, if $\left<Ax,x\right>\geq 0$ for all $x\in\mathbb{C}^n.$ If, for all nonzero $x\in\mathbb{C}^n$, we have $\left<Ax,x\right>>0$, then $A$ is said to be positive definite, written as $A>O$. 

Given $A\in\mathcal{M}_n$, the numerical range $W(A)$ of $A$ is defined as the image of the unit sphere of $\mathbb{C}^n$ under the quadratic form $x\mapsto \left<Ax,x\right>.$ That is,
\[W(A)=\left\{\left<Ax,x\right>:x\in\mathbb{C}^n, \|x\|=1\right\}.\]
The numerical range has been an important notion in matrix analysis, with numerous applications and considerable interest in the literature. Studies of the numerical range cover all possible related notions, such as its shape in the complex plane \cite{bebiano2023numerical,chien1998geometric,keeler1997numerical}, geometric properties such as compactness and convexity \cite{ballantine1978numerical,davis1971toeplitz,parker1951characteristic}, inclusion of eigenvalues \cite{f2008numerical}, and inclusion of the origin \cite{gau2008numerical}.

The concern of whether $W(A)$ includes the origin or not has received the attention of researchers due to the following interesting property \cite{johnson1972matrices}.

\begin{proposition}
    Let $A\in\mathcal{M}_n$ be such that $0\not\in W(A).$ Then there exists $\theta\in\mathbb{R}$ such that 
    \[W\left(e^{i\theta}A\right)\subseteq \{z\in\mathbb{C}: \Re z>0\},\]
    where $\Re z=\frac{z+z^*}{2}$ is the real part of $z$.
\end{proposition}

Due to this property, an accretive matrix was defined as a matrix with numerical range in 
$\Gamma:=\{z\in\mathbb{C}: \Re z\geq 0\}.$ For simplicity, we will write $A\in \Gamma_n$ to mean that $A\in\mathcal{M}_n$ is accretive.

Letting $\Re(A)=\frac{A+A^*}{2}$ and $\Im(A)=\frac{A-A^*}{2i}$ be the real and imaginary parts of $A$, respectively, we have 
\[A\in\Gamma_n \Leftrightarrow \Re(A)\geq O.\]
If $\Re(A)>O$, we write $A\in\Gamma_n^+,$ to mean that $A$ is strictly accretive. We remark that in the literature, accretive matrices were defined as those matrices $A$ with $\Re(A)\geq O$, sometimes and as those with $\Re(A)>O.$ This is just a conventional matter.

\begin{remark}\label{rem_1}
    If $A\in\Gamma_n,$ we may define $A_k=A+\frac{1}{k}I$ as a sequence of matrices such that $A_k\in\Gamma_n^+,$ and $A_k\to A.$ Thus, elements of $\Gamma_n$ are limits, in the  norm topology, of elements from $\Gamma_n^+.$ This remark becomes handy when we need to deal with the inverse of an accretive matrix. Elements of $\Gamma_n^+$ are invertible, while some elements of $\Gamma_n$ are not.
\end{remark}

Related to accretive matrices, the notion of sectorial matrices is used as an alternative terminology. If we define, for $\alpha\in \left[0,\frac{\pi}{2}\right),$
\[S_{\alpha}= \lbrace 
 z\in \mathbb{C}: \Re(z)\geq 0 \text{ and } \vert \Im(z)\vert \leq \tan(\alpha) \Re(z)\rbrace,\]
 then it can be seen that $S_{\alpha}$ is a sector in the right half plane, with half angle $\alpha$. It can be easily shown that 
 \[A\in\Gamma_n \Leftrightarrow W(A)\in S_{\alpha}\;for\;some\;\alpha\in \left[0,\frac{\pi}{2}\right).\]
For $\alpha\in [0,\pi/2)$, we use the notation $\Pi_{n,\alpha}$ to denote all matrices in $\mathcal{M}_n$ with numerical range in $S_{\alpha}.$ Thus, it is evident that
\[\Gamma_n=\bigcup_{\alpha\in [0,\pi/2)}\Pi_{n,\alpha}.\]

The class of accretive matrices has recently gained a great attention. We refer the reader to \cite{kato1961fractional,arlinskii2002sectorial,bedrani2021positive,bedrani2021numerical,bedrani2021weighted,drury2015principal,drury2014singular,furuichi2024further,lin2016some,raissouli2017relative,sababheh2022operator,tan2020extension,zhang2015matrix, nasiri2022new,yang2022some} as a list of such references.

%%%%%%%%%%%%%%%%%%%%%%%%%%%%%%%%%%%%%%%%%%%%%%%%%%%%%%%%%%%%%%%%%%%%%%%%%%%%%%%%%%%%%%%%%%%
%%%%%%%%%%%%%%%%%%%%%%%%%%%%%%%%%%%%%%%%%%%%%%%%%%%%%%%%%%%%%%%%%%%%%%%%%%%%%%%%%%%%%%%%%%%%%%
For the \(2 \times 2\) block matrix \(M= \begin{bmatrix}
        A & X\\
        X^*& B
    \end{bmatrix}\), its partial transpose is defined by \(M^{\tau} = \begin{bmatrix}
        A & X^*\\
        X& B
    \end{bmatrix}\). It is well known that if \(M\geq O\), then \(M^{\tau}\geq O\) need not be true \cite[p. 221]{ghaemi2021advances}. We say that \(M\) is positive partial transpose (PPT for short) if \(M\geq O\) and \(M^{\tau}\geq O\). The class of PPT matrices has been a renowned topic with significance in mathematical physics. We refer the reader to \cite{alakhrass2021note,alakhrass2023singular,gumus2022positive,lin2015inequalities,lin2015hiroshima,sababheh2024further} for some recent progress on PPT matrices, and to \cite{horodecki1996necessary,peres1996separability} for significance of this class in physics.
    
    Motivated by the recent treatment of accretive matrices, accretive partial transpose (APT) matrices have been defined recently in  \cite{kuai2018extension} to extend the concept of PPT matrices. The \(2 \times 2\) block matrix \(M= \begin{bmatrix}
        A & X\\
        Y^*& B
    \end{bmatrix}\) is called accretive partial transpose if \(M\) and \(M^{\tau}\) are accretive. Clearly, the class of APT matrices includes that  of PPT matrices. 

    One can see that if \(M\) is accretive, then so is \(\begin{bmatrix}
   B&Y^*\\
   X&A
\end{bmatrix}\), by conjugating  with \(\begin{bmatrix}
    O&I\\
    I&O
\end{bmatrix}\).\\

We refer the reader to \cite{huang2022refining,huang2023singular,kuai2018extension,liu2021inequalities,yang2022inequalities} for  properties and inequalities related to  APT matrices.

%%%%%%%%%%%%%%%%%%%%%%%%%%%%%%%%%%%%%%%%%%%%%%%%%%%%
%%%%%%%%%%%%%%%%%%%%%%%%%%%%%%%%%%%%%%%%%%%%%%%%%%%%%%%%%%%%%%%%%%%
%%%%%%%%%%%%%%%%%%%%%%%%%%%%%%%%%%%%%%%%%%%%%%%%%%%%%%%%%%%%%%%%%%%%%%%%%%%%%%%%%%%%

Given $A\in\mathcal{M}_n$, if all eigenvalues of \(A\) are real, we use the notation $\lambda_j(A)$ to denote the $j-$th eigenvalue.   The singular values of \(A\in\mathcal{M}_n\) are the eigenvalues of \(\vert A\vert = (A^*A)^{1/2}\), and $s_j(A)$ will denote the $j-$th singular value of $A$ . A unitarily invariant  norm \(\Vert \cdot\Vert_u\) on $\mathcal{M}_n$ is a matrix norm that satisfies the additional property   \(\Vert UAV \Vert_u = \Vert A\Vert_u\) for all $A,U,V\in\mathcal{M}_n$ such that $U$ and $V$ are unitaries. In the sequel, when we write $\|\cdot\|_u$, we mean an arbitrary unitarily invariant norm on $\mathcal{M}_n$, normalized so that 
$\|{\text{diag}}(1,0,\ldots,0)\|=1$. Of particular interest, the spectral norm (or the usual operator norm) is denoted by $\|\cdot\|.$
%%%%%%%%%%%%%%%%%%%%%%%%%%%%%%%%%%%%%%%%%%%%%%%%%%%%%%%%%%%%%%%%%%%%%%%%%%%%%%%%%%%%%%%%%%%%
%%%%%%%%%%%%%%%%%%%%%%%%%%%%%%%%%%%%%%%%%%%%%%%%%%%%%%%%%%%%%%%%%%%%%%%%%%%%%%%%%%%%%%%%%%%%%%
%%%%%%%%%%%%%%%%%%%%%%%%%%%%%%%%%%%%%%%%%%%%%%%%%%%%%%%%%%%%%%%%%%%%%%%%%%%%%%%%%%%%%%%%%%%%%%%

The main purpose of this paper is to prove new properties of APT matrices, in a way that extends some known facts about PPT ones. However, to state and prove our results, we will need some results and notions from the literature, as discussed in the next section.
\section{Preliminaries}
In this section, we discuss briefly some needed notions and facts that are known to researchers in this field. A familiar reader with the topic can proceed to the main results section immediately. The organization of the stated results in this section will be as follows. We begin with some results about positivity of the block form $M$, then some results about accretive matrices will be presented. After that results on PPT and APT matrices will be stated.

\begin{lemma}(\cite[Theorem 1.3.3]{bhatia2009positive})\label{lemma_pos_inv}
    Let \(A,B, X\in \mathcal{M}_n \) be such that \(A,B>O\). Then \(M=\begin{bmatrix}
    A & X\\
    X^* &  B
\end{bmatrix} \geq O, \) if and only if \(A\geq X B^{-1 } X^*\).
\end{lemma}
As an interesting contribution, we prove the accretive version of Lemma \ref{lemma_pos_inv} in Theorem \ref{thm_accretive_equiv} below.

%%%%%%%%%%%%%%%%%%%%%%%%%%%%%%%%%%%%%%%%%%%%%%%%%%%%%%%%%%%%%%%%%%%%%%%%%%%%%%%%%
%%%%%%%%%%%%%%%%%%%%%%%%%%%%%%%%%%%%%%%%%%%%%%%%%%%%%%%%%%%%%%%%%%%%%%%%%%%%%%%%%%
%%%%%%%%%%%%%%%%%%%%%%%%%%%%%%%%%%%%%%%%%%%%%%%%%%%%%%%%%%%%%%%%%%%%%%%%%%%%%%%%%%%%%
\begin{lemma}\label{lemma_pos_bl}
     Let \(A,B, X\in \mathcal{M}_n \) be such that   \(M= \begin{bmatrix}
        A & X\\
        X^*& B
    \end{bmatrix}\geq O\).
    \begin{enumerate}[(i)]
        \item \cite{bourin2013positive} For any unitarily invariant norm $\|\cdot\|_{u}$ on $\mathcal{M}_n$, 
        \[ \Vert M\Vert_u \leq \Vert A+B\Vert_u .\]
          \item \cite{bourin2013positive} If \(X\) is Hermitian, then there exist two unitary matrices $U,V\in\mathcal{M}_n$ such that \[M= \dfrac{1}{2}\Big ( U(A+B)U^* + V(A+B)V^*\Big).\] 
          \item \label{Tao's singular values ineq}\cite{tao2006more} For $j=1,\ldots,n$,
          \[2s_j(X)\leq s_j(M).\]
    \end{enumerate}
\end{lemma}

The notion of matrix means is indeed essential in studying positivity of block matrices. We recall that the weighted geometric mean of the positive definite matrices \(A\) and \(B\) is defined by the equation 
\begin{equation}
    A\sharp_t B = A^{1/2} \Big(A^{-1/2} B A^{-1/2} \Big)^t A^{1/2}, \; 0\leq t\leq 1.
\end{equation}
When \(t=\frac{1}{2}\), we simply write \(A\sharp B\). The weighted geometric mean is a special case of matrix means. Recall that a matrix mean on the set of all positive definite matrices is a binary operation $\sigma$ defined by 
\begin{equation}\label{defintion of op mean}
        A \sigma B = A^{1/2} \, f(A^{-1/2} B A^{-1/2})\, A^{1/2},
\end{equation} 
where \(f\in \mathfrak{m}:= \lbrace f:(0,\infty)\longrightarrow (0,\infty): f {\text{ is an operator monotone function with } } f(1)=1\rbrace\). By an operator monotone function \(f:(0,\infty)\to (0,\infty)\), we mean a function that satisfies \(f(A)\leq f(B)\) whenever \(O<A\leq B\). The geometric mean above corresponds to the function \(f(x)=x^{t}, 0\leq t\leq 1\). Another important mean is the weighted arithmetic mean defined for \(A,B\geq O\) as 
\(A\nabla_t B= (1-t)A+t B,  0\leq t\leq 1.\) For background on matrix means, we refer the reader to \cite{kubo1980means}. Although the above definition for matrix means is stated for positive definite matrices, it is still valid for positive semi-definite matrices via a limit approach. The geometric mean enjoys the following properties \cite{pusz1975functional}.

\begin{lemma}\label{w-g-mean of positive matrices}
Let \(A,B,C,D \geq O\), and let \(t\in [0,1]\). Then 
\begin{enumerate}[(i)]
   \item \(A \sharp_{t} B\geq O.\). 
    \item \(A \sharp_{1-t} B= B \sharp_t A\).
    \item \label{monotonocity of the w-g-m} \(A \sharp_{t} B \leq C \sharp_t D\), if \(A\leq C\) and \(B\leq D\).
\end{enumerate}
\end{lemma}

The notion of matrix means was extended to accretive matrices in \cite{bedrani2021positive}, where the definition in \eqref{defintion of op mean} applies to strictly accretive matrices.
%%%%%%%%%%%%%%%%%%%%%%%%%%%%%%%%%%%%%%%%%%%%%%%%%%%%%%%%%%%%%%%%%%%%%%%%%%%%%%%%%%%%%%%%
%%%%%%%%%%%%%%%%%%%%%%%%%%%%%%%%%%%%%%%%%%%%%%%%%%%%%%%%%%%%%%%%%%%%%%%%%%%%%%%%%%%%%%%%%%%%%

For a non-zero matrix mean \(\sigma\), the adjoint \(\sigma^*\) is defined by \cite{kubo1980means}  \[A \sigma^* B = \Big(A^{-1} \sigma B^{-1} \Big)^{-1}.\]

One can easily verify that for any strictly accretive matrices \(A\) and \(B\), and for any invertable matrix \(X\), we have \cite[Theorem 5.4]{bedrani2021positive}
\[X^* (A \sigma B) X=( X^* A X) \sigma (X^*B X),\]
and for any  \(A,B\geq O\) and for any matrix \(X,\), we have \cite{kubo1980means} 
\begin{equation}\label{eq_conj_sig}
  X^* (A \sigma B )X \leq (X^* A X) \sigma (X^* B X).  
\end{equation}

For the next result, we clarify one point. When $f\in\mathfrak{m},$ it is defined on $(0,\infty).$ However, it is known that such $f$ has an analytic continuation to $\mathbb{C}\backslash (-\infty,0],$ see  \cite[Theorem V.4.7]{bhatia2013matrix}. See also \cite[Proposition 1.2]{bedrani2021positive} for further details. When $A\in\Gamma_n^+$, we know that the spectrum of $A$ avoids $(-\infty,0],$ which makes $f(A),$ for such $A,$ well defined.

Moreover, by appealing to a limit argument as in Remark \ref{rem_1}, we can see how to pass from $\Gamma_n^+$ to $\Gamma_n$ for most results.  In what follows, $\Pi_{n,\alpha}^{+}=\Gamma_n^+\cap \Pi_{n,\alpha}.$
\begin{lemma}\cite[Proposition 7.1, Proposition 7.2]{bedrani2021positive}\label{Copmarision bw R(f(A)) and f(R(A))}
    Let \(A\in \Pi_{n,\alpha}^{+}\) for some \(\alpha\in [0,\pi/2)\)  and let \(f \in \mathfrak{m}\). Then \[\cos^2(\alpha) \Re(f(A))\leq f(\Re(A))\leq \Re(f(A)). \]
    Moreover \cite{drury2014singular,lin2015extension},
    \[\Re(A^{-1})\leq (\Re(A))^{-1}\leq \sec^2\alpha\,\Re(A^{-1}).\]
\end{lemma}

\begin{lemma}\cite{bedrani2021positive}\label{accretivity of A sigma B}
    Let \(A,B\in \Pi_{n,\alpha}^{+}\) for some \(\alpha\in [0,\pi/2)\) and let $\sigma$ be any matrix mean. Then 
  \[\Re(A) \sigma \Re(B) \leq \Re(A \sigma B) \leq \sec^2(\alpha) (\Re(A) \sigma \Re(B)).\]  
  So, if $A,B\in\Gamma_n^+$, then so does \(A\sigma B\).
\end{lemma}

%%%%%%%%%%%%%%%%%%%%%%%%%%%%%%%%%%%%%%%%%%%%%%%%%%%%%%%%%%%%%%%%%%%%%%%%%%%%%%%%
%%%%%%%%%%%%%%%%%%%%%%%%%%%%%%%%%%%%%%%%%%%%%%%%%%%%%%%%%%%%%%%%%%%%%%%%%%%%%
%%%%%%%%%%%%%%%%%%%%%%%%%%%%%%%%%%%%%%%%%%%%%%%%%%%%%%%%%%%%%%%%%%%%%%%%%%%%%%

In addition to the above lemmas, we have the following results about PPT matrices. The first result is usually referred to as  Hiroshima's inequality \cite{hiroshima2003majorization,lin2015hiroshima}.
\begin{lemma}\label{Hiroshima's ineq}
  Let \(A,B, X\in \mathcal{M}_n \) be such that  \(M=\begin{bmatrix}
    A & X\\
    X^* &  B
\end{bmatrix} \) is PPT. Then  \(\Vert X\Vert_u \leq \Vert A+B\Vert_u. \)  
\end{lemma}

\begin{lemma} \cite[Theorem 2.1]{gumus2022positive}\label{w-g-mean is PPT}
    Let \(A,B, X\in \mathcal{M}_n \) be such that \(A,B\geq O\). Then \(M= \begin{bmatrix}
    A & X\\
    X^* & B
\end{bmatrix}\) is PPT if and only if \(\begin{bmatrix}
    A\sharp_tB & X\\
    X^* & A\sharp_{1-t} B
\end{bmatrix} \) is PPT for all \(t\in [0,1]\). 
\end{lemma}
The APT version of Lemma \ref{w-g-mean is PPT} is stated in Theorem \ref{w-g-mean is APT} below. 

\begin{lemma} \cite[Theorem 2.1, Corollary 2.1]{sababheh2024further}\label{Absolute value of X ineq}
 Let \(A,B, X\in \mathcal{M}_n \) be such that \(M= \begin{bmatrix}
    A & X\\
    X^* & B
\end{bmatrix}\) is PPT, and  let \(X= U \vert X\vert\) be the polar decomposition of \(X\). Then for any \(0\leq t\leq 1\), \[\vert X\vert \leq \Big(A\sharp_t\, (U^* B U) \Big) \sharp \Big(  A\sharp_{1-t}\,(U^* B U)  \Big),\]
and
\[\vert X^*\vert \leq \Big( (U AU^*)\sharp_{t}\, B \Big)\sharp \Big((UAU^*)\sharp_{1-t}\,B \Big). \]
Furthermore,
 \[\vert X\vert \leq \Big(A\sharp_t\, B \Big) \sharp \Big( U^* (A\sharp_{1-t}\, B) U  \Big),\]
and
\[\vert X^*\vert \leq \Big( U (A\sharp_{t}\, B) U^*\Big)\sharp \Big(A\sharp_{1-t}\,B \Big). \]
\end{lemma}

\begin{lemma}\cite[Corollary 2.2]{sababheh2024further} \label{eigenvalues of the abs value of X}
    Let \(M= \begin{bmatrix}
    A & X\\
    X^* & B
\end{bmatrix}\) be PPT with \(A,B, X\in \mathcal{M}_n \). Then for \(j= 1,2..., n\)  and \(0\leq t\leq 1\),
\[\lambda_j(2\vert X\vert - A\sharp_t B)\leq \lambda_j(A \sharp_{1-t} B).\] 
\end{lemma}
%%%%%%%%%%%%%%%%%%%%%%%%%%%%%%%%%%%%%%%%%%%%%%%%%%%%%%%%%%%%%%%%%%%%%%%%%%%%%%%%
%%%%%%%%%%%%%%%%%%%%%%%%%%%%%%%%%%%%%%%%%%%%%%%%%%%%%%%%%%%%%%%%%%%%%%%%%%%%%%%%
We refer the reader to Corollary \ref{absolute value of X+Y ineq}, Corollary \ref{absolute value of X+Y ineq-2} and Corollary \ref{e.values of (2|X|-R(w-g-mean)) and the s.values of (w-g-mean)} for the APT versions of  Lemma \ref{Absolute value of X ineq} and Lemma \ref{eigenvalues of the abs value of X}.
%%%%%%%%%%%%%%%%%%%%%%%%%%%%%%%%%%%%%%%%%%%%%%%%%%%%%%%%%%%%%%%%%%%%%%%%%%%%%%%%
%%%%%%%%%%%%%%%%%%%%%%%%%%%%%%%%%%%%%%%%%%%%%%%%%%%%%%%%%%%%%%%%%%%%%%%%%%%%%%%%
 \begin{lemma} \cite[Theorem 2.3]{gumus2022positive}\label{positivity of sigma(M)}
     Let \(M= \begin{bmatrix}
    A & X\\
    X^* & B
\end{bmatrix}\) be PPT with \(A,B, X\in \mathcal{M}_n, \) and let \(\sigma\) be a matrix mean. Then 
\(\begin{bmatrix}
    A\sigma^* B & X\\
    X^* & B \sigma A
\end{bmatrix} \) is PPT. 
 \end{lemma}
  
\begin{lemma} \cite[Theorem 2.8]{gumus2022positive}\label{positivity of f(M)}
    Let \(f:[0,\infty] \longrightarrow (0,\infty)\) be an operator concave function and let \(\begin{bmatrix}
        f(A)  & X\\
        X^* & f(B)
    \end{bmatrix}\) be PPT. Then \(\begin{bmatrix}
        f(A)\nabla_t f(B) & X\\
        X^* & f(A \nabla_tB )
    \end{bmatrix}\) is PPT for any \(0\leq t\leq 1\).
\end{lemma}
%%%%%%%%%%%%%%%%%%%%%%%%%%%%%%%%%%%%%%%%%%%%%%%%%%%%%%%%%%%%%%%%%%%%%%%%%%%%%%%%%%%%%%%
%%%%%%%%%%%%%%%%%%%%%%%%%%%%%%%%%%%%%%%%%%%%%%%%%%%%%%%%%%%%%%%%%%%%%%%%%%%%%%%%%%%%%%%%%%
The APT versions of Lemma \ref{positivity of sigma(M)} and Lemma \ref{positivity of f(M)} are stated in Theorem \ref{accretivity of sigma(M)} and Theorem \ref{accretivity of f(M)} below.
%%%%%%%%%%%%%%%%%%%%%%%%%%%%%%%%%%%%%%%%%%%%%%%%%%%%%%%%%%%%%%%%%%%%%%%%%%%%%%%%%%%%%%%
%%%%%%%%%%%%%%%%%%%%%%%%%%%%%%%%%%%%%%%%%%%%%%%%%%%%%%%%%%%%%%%%%%%%%%%%%%%%%%%%%%%%%%%%

Finally, we state some results for APT matrices.
\begin{lemma}\label{Zhang-Norm of M vs Norm of the diagonal entries}
    Let \(M=\begin{bmatrix}
        A & X\\
        Y^*& B
    \end{bmatrix}\) be APT with \(A,B\in\Gamma_n^{+}, X, Y\in \mathcal{M}_n \). 
    \begin{enumerate}[(i)]
        \item \cite[Theorem 3.5]{zhang2015matrix}   If \(M\in\Pi_{2n,\alpha}\) for some \(\alpha\in [0,\pi/2)\), then 
    \(\Vert M\Vert_u \leq \sec(\alpha) \Vert A+B\Vert_u.\)
    \item \label{g-m is APT}\cite{liu2021inequalities} The block matrix 
    $\begin{bmatrix}
        A\sharp B & X\\
        Y^*& A\sharp B
    \end{bmatrix}$  is also APT.
    \end{enumerate}
  
\end{lemma}

\section{Main Results}
In this section, we present our results, where we extend some of the known properties and inequalities in the PPT case to the APT case.\\
We point out that upon letting $\alpha=0$, all stated results for APT matrices below reduce to known scenarios for PPT matrices.\\

Before proceeding, we state and prove the following version of Lemma \ref{lemma_pos_inv} for accretive blocks. The motivation of this result is the observation that $A\geq XB^{-1}X^*$ is equivalent to $A-XB^{-1}X^{*}\geq O.$ We point out that the first assertion of this theorem was shown in \cite{drury2014singular}.

\begin{theorem}\label{thm_accretive_equiv}
    Let $A,B\in\Gamma_n^{+}, X\in\mathcal{M}_n.$ If $M=\begin{bmatrix}
    A & X\\
    X^* &  B
\end{bmatrix}\in\Gamma_n, $ then $A-XB^{-1}X^*\in\Gamma_n.$ \\
On the other hand, if $B\in\Pi_{n,\alpha}$ for some $\alpha\in [0,\pi/2),$ is such that $A-XB^{-1}X^*\in\Gamma_n,$ then $\left[\begin{matrix}A&\cos\alpha X\\\cos\alpha X^*&B\end{matrix}\right]\in\Gamma_n.$
\end{theorem}
\begin{proof}
    If $M\in\Gamma_n$, then $\Re(M)\geq O$. Since $\Re(M)=\begin{bmatrix}
    \Re(A) & X\\
    X^* &  \Re(B)
\end{bmatrix},$ Lemma \ref{lemma_pos_inv} implies that $\Re(A)\geq X(\Re(B))^{-1}X^*.$ But it is well known that when $B\in\Gamma_n^{+}$, then $(\Re(B))^{-1}\geq \Re(B^{-1});$ see Lemma \ref{Copmarision bw R(f(A)) and f(R(A))}. Consequently,
$\Re(A)\geq X(\Re(B))^{-1}X^*\geq X\Re(B^{-1})X^*=\Re(XB^{-1}X^*),$ which ensures that
$\Re(A-XB^{-1}X^*)\geq O.$ This is equivalent to saying that $A-XB^{1}X^*\in\Gamma_n.$\\
On the other hand, assume that $A-XB^{-1}X^*\in\Gamma_n.$ This means that $\Re(A-XB^{-1}X^*)\geq O,$ or $\Re(A)-X\Re(B^{-1})X^*\geq O.$ But we know that $\Re(B^{-1})\geq \cos^{2}\alpha (\Re(B))^{-1}$ from Lemma \ref{Copmarision bw R(f(A)) and f(R(A))}. Therefore, we have $\Re(A)-\cos^{2}\alpha X(\Re(B))^{-1}X^*\geq O,$ which is equivalent to saying that $\left[\begin{matrix}\Re(A)&\cos\alpha X\\\cos\alpha X^*&\ \Re(B)\end{matrix}\right]\geq O.$ This last statement ensures that $\left[\begin{matrix}A&\cos\alpha X\\\cos\alpha X^*&B\end{matrix}\right]\in\Gamma_n.$
\end{proof}
%%%%%%%%%%%%%%%%%%%%%%%%%%%%%%%%%%%%%%%%%%%%%%%%%%%%%%%
%%%%%%%%%%%%%%%%%%%%%%%%%%%%%%%%%%%%%%%%%%%%%%%%%%%%%%%%%%%%%%%%%%%
%%%%%%%%%%%%%%%%%%%%%%%%%%%%%%%%%%%%%%%%%%%%%%%%%%%%%
Notice that since \(A-(A^{-1})^{-1}\geq O\) for any invertable \(A\in M_n\), then by Theorem \ref{thm_accretive_equiv}, we have the following
\begin{corollary}\label{the blk matrix A, cos(alpha) and inv(A) is accretive}
    Let \(A\in \Gamma_n^+\). Then \(\begin{bmatrix}
    A & \cos (\alpha)I\\
    \cos(\alpha)I &  A^{-1}
\end{bmatrix} \in\Gamma_n\).
\end{corollary}

\begin{remark}
    If \(A\in M_n\) is positive definite \((\alpha=0)\), then 
    \(\begin{bmatrix}
        A & I\\
        I& A^{-1}
    \end{bmatrix} \geq O\), which is a well-known fact \cite[page15]{bhatia2009positive}. On the other hand, since \(\Re(A)\leq (\Re(A^{-1}))^{-1}\) for any strictly accretive \(A\in M_n\), the block matrix \(\begin{bmatrix}
    A & I\\
    I &  A^{-1}
\end{bmatrix} \) is never accretive for any strictly accretive non-Hermitian matrix \(A\).
\end{remark}

The following theorem extends part \eqref{g-m is APT} of Lemma \ref{Zhang-Norm of M vs Norm of the diagonal entries}.
\begin{theorem}\label{w-g-mean is APT}
    Let \(A,B\in\Gamma_n^{+},X,Y\in \mathcal{M}_n\). Then \(M=\begin{bmatrix}
        A & X\\
        Y^*& B
    \end{bmatrix}\) is APT if and only if
    \(\begin{bmatrix}
        A\sharp_tB & X\\
        Y^* & A\sharp_{1-t}B
    \end{bmatrix}\)is APT for all \(t\in [0,1]\).
\end{theorem}
\begin{proof}
Assume first that \(M=\begin{bmatrix}
        A & X\\
        Y^*& B
    \end{bmatrix}\) is APT, and set \(Z= \dfrac{X+Y}{2}\). 
  Since \(M\) is APT, then   \(\begin{bmatrix}
        A & X\\
        Y^*& B
    \end{bmatrix}, \begin{bmatrix}
   A&Y^*\\
   X&B
\end{bmatrix}\in \Gamma_{2n}\). Conjugating these two accretive matrices with $\begin{bmatrix}
   O&I\\
   I&O
\end{bmatrix}$ implies that 
\(\begin{bmatrix}
        B & Y^*\\
        X& A
    \end{bmatrix}, \begin{bmatrix}
   B&X\\
   Y^*&A
\end{bmatrix}\in \Gamma_{2n}\).

%%%%%%%%%%%%%%%%%%%%%%%%%%%%%%%%%%%%%%%%%%%%%%%%%%%%%%%%%%%%%%%%%%%
%%%%%%%%%%%%%%%%%%%%%%%%%%%%%%%%%%%%%%%%%%%%%%%%%%%%%%%%%%%%%%%%%%%%%%%%%%%%%%%%%%%%%%%%%%%%%%
Now,
\[\begin{bmatrix}
        A & X\\
        Y^*& B
    \end{bmatrix}, \begin{bmatrix}
   B&X\\
   Y^*&A
\end{bmatrix}\in \Gamma_{2n}\Rightarrow \begin{bmatrix}
   \Re(A) & Z\\
   Z^*& \Re(B)
\end{bmatrix}, \begin{bmatrix}
   \Re(B) & Z\\
   Z^*& \Re(A)
\end{bmatrix}\geq  O.\]
%%%%%%%%%%%%%%%%%%%%%%%%%%%%%%%%%%%%%%%%%%%%%%%%%%%%%%%%%%%%%%%%%%%%%%%%%%%%%%%%%%%%%%%%%
%%%%%%%%%%%%%%%%%%%%%%%%%%%%%%%%%%%%%%%%%%%%%%%%%%%%%%%%%%%%%%%%%%%%%%%%%%%%%%%%%%%%%%%%%%%%
Applying Lemma \ref{lemma_pos_inv}, we obtain
\begin{equation}\label{Real A as an upper bound}
    Z \Re(B)^{-1} Z^*\leq \Re(A)\; \text{ and }\; Z \Re(A)^{-1} Z^*\leq \Re(B).
\end{equation}\\
Similarly, 
\[\begin{bmatrix}
   A&Y^*\\
   X&B
\end{bmatrix},\begin{bmatrix}
        B & Y^*\\
        X& A
    \end{bmatrix}\in\Gamma_{2n}\Rightarrow \begin{bmatrix}
   \Re(A) & Z^*\\
   Z& \Re(B)
\end{bmatrix}, \begin{bmatrix}
   \Re(B) & Z^*\\
   Z& \Re(A)
\end{bmatrix}\geq O. \]
Lemma \ref{lemma_pos_inv} again implies
\begin{equation}\label{Real B as an upper bound}
 Z^* \Re(B)^{-1} Z \leq \Re(A)\; \text{ and } \; Z^* \Re(A)^{-1} Z \leq \Re(B).   
\end{equation}
But for any \(t\in [0,1]\), Lemma \ref{accretivity of A sigma B} implies
 \[\Re(A\sharp_{1-t} B) \geq   \Re(A)\sharp_{1-t} \Re(B).\]
Since both \(\Re(A), \Re(B)>O\), we get 
\begin{align*}
    Z \Big( \Re(A\sharp_{1-t} B)\Big)^{-1} Z^* & \leq  Z  \Big( \Re(A)\sharp_{1-t}\; \Re(B) \Big)^{-1} Z^*\\
    & = Z \Big( \Re(A)^{-1}\sharp_{1-t}\; \Re(B)^{-1}\Big) Z^*\\
    & =  Z \Big(  \Re(B)^{-1}\sharp_{t}\; \Re(A)^{-1}\Big) Z^*\\
    & \leq  \Big(Z \Re(B)^{-1} Z^* \sharp_{t}\; Z \Re(A)^{-1} Z^* \Big)\quad ({\text{by}}\; \eqref{eq_conj_sig})\\
   &  \leq   \Re(A) \sharp_t\; \Re(B) \quad (\text{by }\,  \eqref{Real A as an upper bound}\;{\text{and\;Lemma}} \;\ref{w-g-mean of positive matrices}) \\
   & \leq \Re(A\sharp_t B)\quad ({\text{by\;Lemma}}\;\ref{accretivity of A sigma B}).\\
\end{align*}
Applying Lemma \ref{lemma_pos_inv}, this is equivalent to 
  \(\begin{bmatrix}
   \Re(A\sharp_t B) & Z\\
   Z^*& \Re(A\sharp_{1-t} B)
\end{bmatrix}\geq O.\) \\ 
Similarly, 
\begin{align*}
    Z^*\Big(\Re(A\sharp_{1-t} B)\Big)^{-1} Z
    & \leq  Z^* \Big(  \Re(A)\sharp_{1-t}\; \Re(B) \Big)^{-1} Z\\
    & =  Z^* \Big( \Re(A)^{-1}\sharp_{1-t}\;  \Re(B)^{-1}\Big) Z\\
    & =  Z^* \Big( ( \Re(B)^{-1}\sharp_{t} \; \Re(A)^{-1}\Big) Z\\
    & \leq  \Big(\,Z^* \Re(B)^{-1} Z \sharp_{t}\;  Z^* \Re(A)^{-1} Z\,\Big) \hspace{50mm}\\
   &\leq \Re(A) \sharp_t\; \Re(B) \, \hspace{60mm}\text{ by }\, \eqref{Real B as an upper bound}\\
   & \leq \Re(A\sharp_t B),\\
\end{align*}
which is equivalent to saying  \(\begin{bmatrix}
   \Re(A\sharp_t B) & Z^*\\
   Z& \Re(A\sharp_{1-t} B)
\end{bmatrix}\geq O.\)
Thus, we have shown that
\[\begin{bmatrix}
   \Re(A\sharp_t B) & Z\\
   Z^*& \Re(A\sharp_{1-t} B)
\end{bmatrix}, \begin{bmatrix}
   \Re(A\sharp_t B) & Z^*\\
   Z& \Re(A\sharp_{1-t} B)
\end{bmatrix}\geq O, \]
which means that
\(\begin{bmatrix}
        A\sharp_tB & X\\
        Y^* & A\sharp_{1-t}B
    \end{bmatrix}\) and \(\begin{bmatrix}
        A\sharp_tB & Y^*\\
        X & A\sharp_{1-t}B
    \end{bmatrix}\) are accretive, and hence \(\begin{bmatrix}
        A\sharp_tB & X\\
        Y^* & A\sharp_{1-t}B
    \end{bmatrix}\) is APT  for any \(t\in [0,1]\). This completes the sufficiency part of the theorem.\\
    
For the necessity part, assume that \(\begin{bmatrix}
        A\sharp_tB & X\\
        Y^* & A\sharp_{1-t}B
    \end{bmatrix}\)is APT for all \(t\in [0,1]\). Then by letting \(t=0\), we reach the desired conclusion because \(A\sharp_0 B= A\) and \(A\sharp_1 B =B\).
\end{proof}
%%%%%%%%%%%%%%%%%%%%%%%%%%%%%%%%%%%%%%%%%%%%%%%%%%%%%%%%%%%%%%%%%%%%%%%%%%%%%%%%%%%%%%
%%%%%%%%%%%%%%%%%%%%%%%%%%%%%%%%%%%%%%%%%%%%%%%%%%%%%%%%%%%%%%%%%%%%%%%%%%%%%%%%%%%%%%
%%%%%%%%%%%%%%%%%%%%%%%%%%%%%%%%%%%%%%%%%%%%%%%%%%%%%%%%%%%%%%%%%%%%%%%%%%%%%%%%%%%%%%%%

%%%%%%%%%%%%%%%%%%%%%%%%%%%%%%%%%%%%%%%%%%%%%%%%%%%%%%%%%%%%%%%%%%%%%%%%%%%%%%%%%%%%%%%%%
%%%%%%%%%%%%%%%%%%%%%%%%%%%%%%%%%%%%%%%%%%%%%%%%%%%%%%%%%%%%%%%%%%%%%%%%%%%%%%%%%%%%%%
%%%%%%%%%%%%%%%%%%%%%%%%%%%%%%%%%%%%%%%%%%%%%%%%%%%%%%%%%%%%%%%%%%%%%%%%%%%%%%%%%%%%%%%%
Notice that if \(M=\begin{bmatrix}
        A & X\\
        X^*& B
    \end{bmatrix}\) is accretive, then \(\Re(M)=\begin{bmatrix}
        \Re(A) & X\\
        X^*& \Re(B)
    \end{bmatrix}\geq O.\) Hence, by part (iii) of Lemma \ref{lemma_pos_bl},
    \[2s_j(X) \leq s_j(\Re(M)).\]
    Since, \(\lambda_j(\Re(A)) \leq s_j(A)\) for any  matrix $A$, one may apply Lemma \ref{Zhang-Norm of M vs Norm of the diagonal entries} to obtain  the following extension of part \eqref{Tao's singular values ineq} of Lemma \ref{lemma_pos_bl} and of Lemma \ref{Hiroshima's ineq}.
\begin{proposition}\label{generalization of Tao's sing.values ineq}
    Let \(A,B,X\in \mathcal{M}_n\) be such that \(M=\begin{bmatrix}
        A & X\\
        X^*& B
    \end{bmatrix}\in\Gamma_{2n}\). Then
    \[2s_j(X) \leq s_j(M),\; j=1,2,...,n.\]

    Moreover, if \(M\) is APT and $M\in\Pi_{2n,\alpha}$ for some \(\alpha\in [0,\pi/2)\), then 
    \begin{equation}\label{generalization of Hiroshima's ineq.}
        2\Vert X \Vert_u \leq \sec(\alpha) \Vert A+B\Vert_u
    \end{equation}
\end{proposition}
\begin{remark}
 It is worth mentioning here that part \eqref{Tao's singular values ineq} of Lemma \ref{lemma_pos_bl} is equivalent to the following well-known fact \[s_j(A-B)\leq s_j(A\oplus B), \; j=1,2,...,n,\]
 when $A,B\geq O.$ Here the direct sum \(A\oplus B \) denotes the block diagonal matrix \(\begin{bmatrix}
        A & O\\
        O& B
    \end{bmatrix}.\) This is no longer true in the accretive case. To see this let \[A= \left[
\begin{array}{cc}
 2+2 i & -1+2 i \\
 3+2 i & 1+i \\
\end{array}
\right]\;{\text{and}}\;B=\left[
\begin{array}{cc}
 1+2 i & -2-i \\
 -2-i & 5+i \\
\end{array}
\right].\] Then it can be seen that both $A,B\in\Gamma_n^+,$ and
\[s(A-B)=\left\{\sqrt{\frac{1}{2} \left(5 \sqrt{97}+61\right)},\sqrt{\frac{1}{2} \left(61-5 \sqrt{97}\right)}\right\}\approx\{7.42443,2.42443\},\]
and
\begin{align*}
    s(A\oplus B)&=\left\{\sqrt{\frac{1}{2} \left(3 \sqrt{165}+41\right)},\sqrt{7 \left(\sqrt{3}+2\right)},\sqrt{7 \left(2-\sqrt{3}\right)},\sqrt{\frac{1}{2} \left(41-3 \sqrt{165}\right)}\right\}\\
    &\approx \{6.30618,5.1112,1.36954,1.11002\}.
\end{align*}
That is, in this example, $s_1(A-B)>s_1(A\oplus B).$
\end{remark}

%%%%%%%%%%%%%%%%%%%%%%%%%%%%%%%%%%%%%%%%%%%%%%%%%%%%%%%%%%%%%%%%%%%%%%%%%%%%%%%%%%%
%%%%%%%%%%%%%%%%%%%%%%%%%%%%%%%%%%%%%%%%%%%%%%%%%%%%%%%%%%%%%%%%%%%%%%%%%%%%%%%%%%
Applying Proposition \ref{generalization of Tao's sing.values ineq}, together with Theorem \ref{w-g-mean is APT}, implies the following.

\begin{corollary}\label{norm X vs norm of the w-g-mean of A and B}
   Let \(A,B\in\Gamma_n^{+},X\in \mathcal{M}_n\) be such that \(M=\begin{bmatrix}
        A & X\\
        X^*& B
    \end{bmatrix}\) is APT with \(M\in\Pi_{2n,\alpha}\) for some \(\alpha\in [0,\pi/2)\). Then, for any \(t\in [0,1]\), 
    \begin{equation}
        \Vert X\Vert _u\leq \dfrac{\sec(\alpha)}{2} \Vert A\sharp_tB + A\sharp_{1-t}B \Vert_u.
    \end{equation}
\end{corollary}
%%%%%%%%%%%%%%%%%%%%%%%%%%%%%%%%%%%%%%%%%%%%%%%%%%%%%%%%%%%%%%%%%%%%%%%%%%%%%%%%%%%
%%%%%%%%%%%%%%%%%%%%%%%%%%%%%%%%%%%%%%%%%%%%%%%%%%%%%%%%%%%%%%%%%%%%%%%%%%%%%%%%%%
\begin{corollary}\label{absolute value of X+Y ineq}
  Let \(A,B\in\Gamma_n^{+},X,Y\in \mathcal{M}_n\) be such that \(M=\begin{bmatrix}
        A & X\\
        Y^*& B
    \end{bmatrix}\) is APT. Then \[ \Big\vert \dfrac{X+Y}{2} \Big\vert \leq \Re\Big(A\sharp_t(V^*BV) \Big) \sharp \Re\Big(A\sharp_{1-t}\, (V^*BV)\Big),\] 
    and \[\Big\vert \dfrac{X^*+Y^*}{2} \Big\vert \leq \Re\Big((VAV^*)\sharp_t\, B \Big) \sharp \Re\Big((VAV^*)\sharp_{1-t}\, B\Big), \]
    where \(V\) is the unitary matrix in the polar decomposition   \(\dfrac{X+Y}{2}= V\Big\vert \dfrac{X+Y}{2}\Big\vert\).  
\end{corollary}
\begin{proof}
Since \(M\) is APT, then \(\Re(M)=\begin{bmatrix}
        \Re(A) & \dfrac{X+Y}{2}\\
        \dfrac{X^*+Y^*}{2} & \Re(B)
    \end{bmatrix}\)is PPT. Hence, by Lemma \ref{Absolute value of X ineq}, 
  \begin{align*}
      \Big\vert \dfrac{X+Y}{2} \Big\vert & \leq \Big(\Re(A)\sharp_t\; (V^*\Re(B)V)\Big) \sharp \Big( \Re(A)\sharp_{1-t}\; (V^*\Re(B)V)\,\Big)\\
      &= \Big(\Re(A)\sharp_t\; \Re(V^* BV)\Big) \sharp \Big( \Re(A)\sharp_{1-t}\; \Re(V^* BV)\,\Big)\\
      &\leq \Re(A\sharp_t\; (V^*BV)) \sharp \Re(A\sharp_{1-t}\, (V^*BV)),\\      
\end{align*}
where we have used Lemma \ref{accretivity of A sigma B} and part (iii) of Lemma \ref{w-g-mean of positive matrices} to obtain the last inequality. Similarly,
\begin{align*}
    \Big\vert \dfrac{X^*+Y^*}{2} \Big\vert & \leq \Big( (V \Re(A)V^*)\sharp_{t}\, \Re(B)\Big) \sharp  \Big((V \Re(A)V^*)\sharp_{1-t} \Re(B))\Big)\\
    &=\Big( \Re(V AV^*)\sharp_{t}\, \Re(B)\Big) \sharp  \Big(\Re(V A V^*)\sharp_{1-t} \Re(B)\Big)\\
    &\leq \Re\Big( (V AV^*)\sharp_{t}\, B\Big) \sharp  \Re\Big( (V A V^*)\sharp_{1-t}\Big),\\
\end{align*}
which completes the proof.
\end{proof}
  %%%%%%%%%%%%%%%%%%%%%%%%%%%%%%%%%%%%%%%%%%%%%%%%%%%%%%%%%%%%%%%%%%%%%%%%%%%%%%%%%%%
%%%%%%%%%%%%%%%%%%%%%%%%%%%%%%%%%%%%%%%%%%%%%%%%%%%%%%%%%%%%%%%%%%%%%%%%%%%%%%%%%%%%%    
\begin{corollary}\label{absolute value of X+Y ineq-2}
  Let \(A,B\in\Gamma_n^{+},X,Y\in \mathcal{M}_n\) be such that \(M=\begin{bmatrix}
        A & X\\
        Y^*& B
    \end{bmatrix}\) is APT. Then 
    \[ \Big\vert \dfrac{X+Y}{2} \Big\vert \leq \dfrac{\Re\Big(A\sharp_tB + V^*(A\sharp_{1-t} B)V\Big)}{2},\] 
    and \[\Big\vert \dfrac{X^*+Y^*}{2} \Big\vert \leq \dfrac{\Re\Big(V\, (A\sharp_tB)\, V^* + (A\sharp_{1-t} B)\Big)}{2},\]
    where \(V\) is the unitary matrix in the polar decomposition \(\dfrac{X+Y}{2}= V\Big\vert \dfrac{X+Y}{2}\Big\vert\).
\end{corollary}
\begin{proof}
Since \(M\) is APT, then \(\Re(M)=\begin{bmatrix}
        \Re(A) & \dfrac{X+Y}{2}\\
        \dfrac{X^*+Y^*}{2} & \Re(B)
    \end{bmatrix}\) is PPT. Hence, by Lemma \ref{Absolute value of X ineq}, 
  \begin{align*}
      \Big\vert \dfrac{X+Y}{2} \Big\vert & \leq \Big(\Re(A)\sharp_t\; \Re(B)\Big) \sharp \Big(V^*\, (\Re(A)\sharp_{1-t}\; \Re(B))\, V\Big)\\
      &\leq \Re(A\sharp_t\; B) \sharp \Big(V^*\, \Re(A\sharp_{1-t} \;B)\, V\Big)\\
      & = \Re(A\sharp_t\; B) \sharp \Re\Big(V^*\, (A\sharp_{1-t} \; B)\, V\Big)\\
      & \leq \dfrac{\Re(A\sharp_t\; B) + \Re\Big(V^*\, (A\sharp_{1-t}\; B)\, V\Big)}{2}\\
      &= \dfrac{\Re\Big((A\sharp_t\; B) + V^*\, (A\sharp_{1-t}\; B)\, V\Big)}{2},\\
  \end{align*}  
  where we have used Lemma \ref{accretivity of A sigma B}, part \eqref{monotonocity of the w-g-m} of Lemma \ref{w-g-mean of positive matrices} and the arithmetic-geometric mean inequality to obtain the last two inequalities. In a similar fashion,
  \begin{align*}
      \Big\vert \dfrac{X^*+Y^*}{2} \Big\vert & \leq \Big(V\, (\Re(A)\sharp_t \Re(B))\, V^*\Big) \sharp  \Big(\Re(A)\sharp_{1-t} \Re(B))\Big)\\
      &\leq \Big(V\, \Re(A\sharp_t B)\, V^*\Big) \sharp  \Re(A\sharp_{1-t} B)\\
      & =  \Re\Big(V (A\sharp_t B)\, V^*\Big) \sharp \Re( A\sharp_{1-t} B)\\
      & \leq \dfrac{\Re\Big(V\, (A\sharp_t B)\, V^*\Big) + \Re(A\sharp_{1-t} B)}{2}\\
      &= \dfrac{\Re\Big(V\, (A\sharp_t B)\, V^* + (A\sharp_{1-t} B)\Big)}{2},\\
  \end{align*} 
  completing the proof.
\end{proof}
 %%%%%%%%%%%%%%%%%%%%%%%%%%%%%%%%%%%%%%%%%%%%%%%%%%%%%%%%%%%%%%%%%%%%%%%%%%%%%%%%%%%
%%%%%%%%%%%%%%%%%%%%%%%%%%%%%%%%%%%%%%%%%%%%%%%%%%%%%%%%%%%%%%%%%%%%%%%%%%%%%%%%%%%%% 
By applying the same argument used in the proof of Corollary \ref{absolute value of X+Y ineq-2} and Weyl's monotonicity principle, we obtain the following result.

\begin{corollary}
    Let \(A,B\in\Gamma_n^{+},X,Y\in \mathcal{M}_n\) be such that \(M=\begin{bmatrix}
        A & X\\
        Y^*& B
    \end{bmatrix}\) is APT. Then 
    \begin{align*}
        s_j\Big( \dfrac{X+Y}{2} \Big) & \leq s_j\Big((\Re(A)\sharp_t \Re(B) )\sharp (V^*(\Re(A)\sharp_{1-t} \Re(B))V)\Big)\\
        &\leq s_j\Big(\Re(A\sharp_t B )\sharp (V^*(\Re(A\sharp_{1-t} B)V)\Big)\\
        & \leq s_j\Big(\dfrac{\Re(A\sharp_tB + V^*(A\sharp_{1-t} B)V}{2}\Big)\\
        &\leq s_j\Big( \dfrac{A\sharp_tB + V^*(A\sharp_{1-t} B)V}{2}\Big)\\
        \end{align*}
\end{corollary}
%%%%%%%%%%%%%%%%%%%%%%%%%%%%%%%%%%%%%%%%%%%%%%%%%%%%%%%%%%%%%%%%%%%%%%%%%%%%%%%%%%%
%%%%%%%%%%%%%%%%%%%%%%%%%%%%%%%%%%%%%%%%%%%%%%%%%%%%%%%%%%%%%%%%%%%%%%%%%%%%%%%%%%%%% 

Extending Lemma \ref{eigenvalues of the abs value of X} to the APT case, we have the following.
\begin{corollary}\label{e.values of (2|X|-R(w-g-mean)) and the s.values of (w-g-mean)}
   Let \(A,B\in\Gamma_n^+,X\in \mathcal{M}_n\) be such that \(M=\begin{bmatrix}
        A & X\\
        X^*& B
    \end{bmatrix}\) is APT. Then, for any \(t\in [0,1]\),
    \[\lambda_j\Big(2\vert X\vert- \Re(A\sharp_t B) \Big)\leq \lambda_j(\Re(A\sharp_{1-t}B))\leq s_j(A\sharp_{1-t}B), j=1,\ldots,n.\] 
\end{corollary}
 \begin{proof}
   By Corollary \ref{absolute value of X+Y ineq-2}, we have 
   \[2\vert X\vert \leq  \Re(A\sharp_t B) + V^*\, \Re(A\sharp_{1-t} B)\, V,\]
   so that \[2\vert X\vert-\Re(A\sharp_t B) \leq V^* \Re (A\sharp_{1-t} B)\, V .\]
   Thus, by Weyl's monotonicity principle,  \begin{align*}
       \lambda_j\Big(2\vert X\vert- \Re(A\sharp_t B)\Big) & \leq \lambda_j\Big(V^* \Re (A\sharp_{1-t} B)\, V \Big)\\
       &= \lambda_j\Big( \Re (A\sharp_{1-t} B) \Big)\\
       &\leq s_j(A\sharp_{1-t} B)).\\
   \end{align*}
   This completes the proof.
 \end{proof}
 %%%%%%%%%%%%%%%%%%%%%%%%%%%%%%%%%%%%%%%%%%%%%%%%%%%%%%%%%%%%%%%%%%%%%%%%%%%%%%%%%%%
%%%%%%%%%%%%%%%%%%%%%%%%%%%%%%%%%%%%%%%%%%%%%%%%%%%%%%%%%%%%%%%%%%%%%%%%%%%%%%%%%%%%%
An interesting version of Corollary \ref{norm X vs norm of the w-g-mean of A and B} is stated next.
\begin{corollary}
  Let \(A,B\in\Gamma_n^{+},X\in \mathcal{M}_n\) be such that \(M=\begin{bmatrix}
        A & X\\
        X^*& B
    \end{bmatrix}\) is APT. Then, for any \(t\in [0,1]\),  
    \[\Vert X \Vert_u \leq \dfrac{1}{2} \Vert(A\sharp_t\; B) + V^*\, (A\sharp_{1-t}\; B)\, V \Vert_u, \]
    where $V$ is the unitary matrix in the polar decomposition $X=V|X|.$
\end{corollary}
\begin{proof}
    Notice that, by Corollary \ref{absolute value of X+Y ineq-2},
\begin{align*}
    2 s_j\Big( \dfrac{X+Y}{2} \Big) &= 2\lambda_j\Big(\Big\vert \dfrac{X+Y}{2}\Big\vert \Big) \\
    & \leq 2\lambda_j\Big( \, \Re(A\sharp_t B) \sharp \Re\Big(V^*\, (A\sharp_{1-t} B)\, V\Big)\,\Big)\\
    &\leq \lambda_j\Big( \, \Re\Big((A\sharp_t B) + V^*\, (A\sharp_{1-t} B)\, V\Big)\, \Big)\\
    & \leq s_j\Big(A\sharp_t B + V^*\, (A\sharp_{1-t} B)\, V\Big).\\
\end{align*}
Letting $Y=X$ completes the proof.
\end{proof}
 %%%%%%%%%%%%%%%%%%%%%%%%%%%%%%%%%%%%%%%%%%%%%%%%%%%%%%%%%%%%%%%%%%%%%%%%%%%%%%%%%%%
%%%%%%%%%%%%%%%%%%%%%%%%%%%%%%%%%%%%%%%%%%%%%%%%%%%%%%%%%%%%%%%%%%%%%%%%%%%%%%%%%%%%%
Now we have the following extension of Lemma \ref{positivity of sigma(M)} to the APT case.

\begin{theorem}\label{accretivity of sigma(M)}
    Let \(A,B\in\Gamma_n^{+},X\in \mathcal{M}_n\) be such that \(M=\begin{bmatrix}
        A & X\\
        X^*& B
    \end{bmatrix}\) is APT. Then \(\begin{bmatrix}
        A\sigma^* B & X\\
        X^* & B\sigma A
    \end{bmatrix}\) is APT, for any matrix mean $\sigma$. 
\end{theorem}
\begin{proof}

Since \(M=\begin{bmatrix}
        A & X\\
        X^*& B
    \end{bmatrix}\) is APT, it follows that \(\begin{bmatrix}
        \Re(A) & X\\
        X^*& \Re(B)
    \end{bmatrix}, \begin{bmatrix}
        \Re(A) & X^*\\
        X& \Re(B)
    \end{bmatrix} \geq O.\)
    Applying Lemma \ref{lemma_pos_inv} on these two block forms implies that
    \begin{equation}\label{eq_ned_1}
    X\Re(B)^{-1}X^*\leq \Re(A), X\Re(A)^{-1}X^*\leq\Re(B).
    \end{equation}
  \(\begin{bmatrix}
        A\sigma^* B & X\\
        X^* & B\sigma A
    \end{bmatrix}\) is accretive if  \(\begin{bmatrix}
        \Re(A\sigma^* B) & X\\
        X^* & \Re(B\sigma A)
    \end{bmatrix}\geq O\). But
    \begin{align*}
        X \Big(\Re(B\sigma A)\Big)^{-1} X^* &\leq  X\Big( \Re(B)\sigma \Re(A)\Big)^{-1} X^* \quad \text{ (by Lemma \ref{accretivity of A sigma B})}\\
        &=   X \Big(\, \Re(B)^{-1}\sigma^* \Re(A)^{-1}\, \Big) X^* \\
        &\leq \Big( X\Re(B)^{-1} X^* \Big) \sigma^* \Big(X \Re(A)^{-1} X^*\Big) \quad ({\text{by}}\;\eqref{eq_conj_sig})\\
        &\leq \Re(A) \sigma^* \Re(B) \quad (\text{by\;\eqref{eq_ned_1}})\\
        & \leq \Re(A \sigma^* B)\quad (\text{by Lemma \ref{accretivity of A sigma B}}).\\
    \end{align*}
    Lemma \ref{lemma_pos_inv} then implies that \(\begin{bmatrix}
        \Re(A\sigma^* B) & X\\
        X^* & \Re(B\sigma A)
    \end{bmatrix}\geq O\), and hence \(\begin{bmatrix}
        A\sigma^* B & X\\
        X^* & B\sigma A
    \end{bmatrix}\) is accretive.    Similarly, one can show that \(\begin{bmatrix}
        A\sigma^* B & X^*\\
        X & B\sigma A
    \end{bmatrix}\) is accretive, which means that \(\begin{bmatrix}
        A\sigma^* B & X\\
        X^* & B\sigma A
    \end{bmatrix}\) is APT..
\end{proof}
 %%%%%%%%%%%%%%%%%%%%%%%%%%%%%%%%%%%%%%%%%%%%%%%%%%%%%%%%%%%%%%%%%%%%%%%%%%%%%%%%%%%
%%%%%%%%%%%%%%%%%%%%%%%%%%%%%%%%%%%%%%%%%%%%%%%%%%%%%%%%%%%%%%%%%%%%%%%%%%%%%%%%%%%%%

The following lemma is needed to complete the proof of the next main result, which discusses a sufficient condition that \(\begin{bmatrix}
        f(A)\nabla_t f(B) & X\\
        X^* & f(A \nabla_tB )
    \end{bmatrix}\) is APT.

\begin{lemma}\label{positivity of the block matrix f(R(M))}
   Let \(A,B,X\in \mathcal{M}_n\) be such that \(A,B\in \Pi_{n,\alpha}^{+}\) for some \(\alpha\in [0,\pi/2)\), and let \(f\in\mathfrak{m}\) be such that \(\begin{bmatrix}
        \cos^2(\alpha) f(A)  & X\\
        X^* & \cos^2(\alpha) f(B)
    \end{bmatrix}\) is APT. Then \(\begin{bmatrix}
        f(\Re(A)) & X\\
        X^* & f(\Re(B) )
    \end{bmatrix}\) is PPT.
\end{lemma}
\begin{proof}
Since the matrix 
\(\begin{bmatrix}
        \cos^2(\alpha) f(A)  & X\\
        X^* & \cos^2(\alpha) f(B)
    \end{bmatrix}\) is APT, it follows that \[\begin{bmatrix}
        \cos^2(\alpha) \Re(f(A))  & X\\
        X^* & \cos^2(\alpha) \Re(f(B))
    \end{bmatrix}, \begin{bmatrix}
        \cos^2(\alpha) \Re(f(A))  & X^*\\
        X & \cos^2(\alpha) \Re(f(B))
    \end{bmatrix} \geq O.\]
    Applying Lemma \ref{lemma_pos_inv} on these two block forms implies that
    \begin{equation}\label{positivity of R(N)-1}
          X (\cos^2(\alpha) \Re(f(B)))^{-1} X^* \leq \cos^2(\alpha) \Re(f(A))
    \end{equation}
    and
    \begin{equation}\label{positivity of R(N)-2}
       X^* (\cos^2(\alpha) \Re(f(B)))^{-1} X \leq \cos^2(\alpha) \Re(f(A))
    \end{equation}
Now, 
\begin{align*}
    X \Big(f(\Re(B))\Big)^{-1} X^* & \leq X \Big(\cos^2(\alpha) \Re(f(B))\Big)^{-1} X^* \quad \text{ (by Lemma \ref{Copmarision bw R(f(A)) and f(R(A))})}\\
    & \leq \cos^2(\alpha) \Re(f(A)) \quad \text{(by \eqref{positivity of R(N)-1}) }\\
    & \leq  f(\Re(A)) \quad \text{(by Lemma \ref{Copmarision bw R(f(A)) and f(R(A))})}.
\end{align*}
Similarly,
\begin{align*}
    X^* \Big(f(\Re(B))\Big)^{-1} X & \leq X^* \Big(\cos^2(\alpha) \Re(f(B))\Big)^{-1} X  \quad \text{ (by Theorem \ref{Copmarision bw R(f(A)) and f(R(A))})}\\
    & \leq \cos^2(\alpha) \Re(f(A))  \quad \text{ (by \eqref{positivity of R(N)-2}) }\\
    & \leq f(\Re(A)) \quad \text{ (by Lemma \ref{Copmarision bw R(f(A)) and f(R(A))})}.
\end{align*}
That is, \(\begin{bmatrix}
        f(\Re(A)) & X\\
        X^* & f(\Re(B) )
    \end{bmatrix}, \begin{bmatrix}
        f(\Re(A)) & X^*\\
        X & f(\Re(B)).
    \end{bmatrix}\geq O\). This implies the desired conclusion.
\end{proof}

\begin{theorem}\label{accretivity of f(M)}
    Let \(A,B,X\in \mathcal{M}_n\) be such that \(A,B\in \Pi_{n,\alpha}^{+}\) for some \(\alpha\in [0,\pi/2)\), and let \(f\in\mathfrak{m}\) be  such that \(\begin{bmatrix}
        \cos^2(\alpha) f(A)  & X\\
        X^* & \cos^2(\alpha) f(B)
    \end{bmatrix}\) is APT. Then \(\begin{bmatrix}
        f(A)\nabla_t f(B) & X\\
        X^* & f(A \nabla_tB )
    \end{bmatrix}\) is APT for any \(0\leq t\leq 1\).
\end{theorem}

\begin{proof}
Since the matrix 
\( \begin{bmatrix}
        \cos^2(\alpha) f(A)  & X\\
        X^* & \cos^2(\alpha) f(B)
    \end{bmatrix}\) is APT, Lemma \ref{positivity of the block matrix f(R(M))} implies that the matrix 
    \(\begin{bmatrix}
        f(\Re(A)) & X\\
        X^* & f(\Re(B) )
    \end{bmatrix}\) is PPT. Consequently, by Lemma \ref{positivity of f(M)}, the matrix 
    \(\begin{bmatrix}
        f(\Re(A))\nabla_t f(\Re(B)) & X\\
        X^* & f(\Re(A \nabla_tB ))
    \end{bmatrix}\) is PPT for any \(t\in[0,1]\), which implies  
    \begin{equation}\label{arith. mean ineq}
        X \Big(f(\Re(A)\nabla_t\, \Re(B))\Big)^{-1} X^* \leq f(\Re(A)) \nabla_t \, f(\Re(B)).   
        \end{equation}
    Now, \(\begin{bmatrix}
        f(A)\nabla_t f(B) & X\\
        X^* & f(A \nabla_tB )
    \end{bmatrix}\) is APT if \(\begin{bmatrix}
        \Re(f(A)\nabla_t f(B)) & X\\
        X^* & \Re(f(A \nabla_tB ))
    \end{bmatrix}\) is PPT. 
    But, 
        \begin{equation*}
        \begin{split}
            X \Big(\Re(f(A\nabla_t\, B))\Big)^{-1} X^* &\leq X \Big(f(\Re(A\nabla_t\, B))\Big)^{-1} X^* \quad \text{ (by  Lemma \ref{Copmarision bw R(f(A)) and f(R(A))})}\\
        &= X \Big(f\Big(\Re(A)\nabla_t\, \Re(B)\Big)\Big)^{-1} X^*\\
        &\leq f(\Re(A)) \nabla_t \, f(\Re(B)) \quad \text{ (by \eqref{arith. mean ineq})}\\
        &\leq \Re(f(A)) \nabla_t\, \Re(f(B)) \quad\text{ (by  Lemma \ref{Copmarision bw R(f(A)) and f(R(A))})}\\
        & = \Re(f(A)\nabla_t\, f(B)).\\
        \end{split}
        \end{equation*}  
Similarly, one can show that 
\[ X^* \Big(\Re(f(A\nabla_t\, B))\Big)^{-1} X\leq \Re(f(A)\nabla_t\, f(B)),\]
which completes the proof.
\end{proof}
%%%%%%%%%%%%%%%%%%%%%%%%%%%%%%%%%%%%%%%%%%%%%%%%%%%%%%%%%%%%%%%%%%%%%%%%%%%%%%%%%%%
%%%%%%%%%%%%%%%%%%%%%%%%%%%%%%%%%%%%%%%%%%%%%%%%%%%%%%%%%%%%%%%%%%%%%%%%%%%%%%%%%%%%%

Let $A,B,X\in\mathcal{M}_n$. If $A,B> O$, and  \(\Vert X\Vert \leq \Vert A\sharp B\Vert\),  we say that Schwarz inequality holds for the triple \((A,B,X)\), in order.\\
In \cite{ando2016geometric}, Ando established certain conditions on the matrices \(A,B,X\in \mathcal{M}_n\) under which the positivity of the block matrix \(M=\begin{bmatrix}
        A & X\\
        X^*& B
    \end{bmatrix}\) ensures that Schwarz inequality holds for \((A,B,X)\), as follows.
\begin{lemma}\cite[Theorem 3.6]{ando2016geometric}]\label{When does Schwarz inequality hold for postive A and B}
   Let \(A,B,X\in \mathcal{M}_n\) be such that $A,B>O$ and that \(M= \begin{bmatrix}
        A & X\\
        X^*& B
    \end{bmatrix}\geq O\). Then Schwarz inequality holds for \((A,B,X)\) if any  of the following conditions holds:
    \begin{enumerate}
        \item[1.] \(AX=XA\).
        \item[2.] \(X^* A^{-1} X= X A^{-1}X^*\).
        \item[3.]  \(\exists k>0: B= kA\).
    \end{enumerate}
\end{lemma}
In Theorem \ref{When does Schwarz inequality hold for accretive A and B} we extend Theorem \ref{When does Schwarz inequality hold for postive A and B} to the accretive case. 

\begin{theorem}\label{When does Schwarz inequality hold for accretive A and B}
   Let \(A,B,X\in\mathcal{M}_n\) be such that \(A\) is normal. Assume further that \(M= \begin{bmatrix}
        A & X\\
        X^*& B
    \end{bmatrix}\) is accretive, with $A,B\in\Gamma_n^+$. Then \(\Vert X\Vert \leq \Vert A\sharp B\Vert\) if one of the following conditions holds:
    \begin{enumerate}
        \item[1.] \(AX=XA\).
        \item[2.] \(X^* A^{-1} X= X A^{-1}X^*\).
        \item[3.]  \(\exists k>0 : B= kA\).
    \end{enumerate}
\end{theorem}
\begin{proof}
    Assume first that \(AX=XA\). Then, as a consequence of Fuglede-Putnam Theorem \cite{fuglede1950commutativity,putnam1951normal}, we infer that \(A^*X=XA^*\). Adding these two relations for $A$ and $X$ implies \(\Re(A)X=X\Re(A)\). Now, since \(M\) is accretive, then \(\Re(M)= \begin{bmatrix}
        \Re(A) & X\\
        X^*& \Re(B)
    \end{bmatrix}\geq O.\) 
    
    Consequently, by Lemma \ref{When does Schwarz inequality hold for postive A and B}, then Lemma \ref{accretivity of A sigma B}, \[\Vert X\Vert \leq \Vert \Re(A)\sharp \Re(B) \Vert \leq \Vert \Re(A\sharp B)\Vert \leq \Vert A\sharp B\Vert.\]
   
    Similarly, if \(X^* A^{-1} X= X A^{-1}X^*\), then \(X^* (\Re(A))^{-1} X= X (\Re(A))^{-1}X^*\).\\
    Now, since \(M\) is accretive, then \(\Re(M)= \begin{bmatrix}
        \Re(A) & X\\
        X^*& \Re(B)
    \end{bmatrix}\geq O.\) Proceeding as before proves the desired conclusion when \(X^* A^{-1} X= X A^{-1}X^*\).\\
    Finally, if $B=kA$ for some $k>0$, then $\Re(B)=k\Re(A)$, and the same argument as above completes the proof.
\end{proof}
%%%%%%%%%%%%%%%%%%%%%%%%%%%%%%%%%%%%%%%%%%%%%%%%%%%%%%%%%%%%%%%%%%%%%%%%%%%%%%%%%%%
%%%%%%%%%%%%%%%%%%%%%%%%%%%%%%%%%%%%%%%%%%%%%%%%%%%%%%%%%%%%%%%%%%%%%%%%%%%%%%%%%%%%%
The following two lemmas are needed for our last result.
\begin{lemma}\cite[Theorem 3.1]{burqan2017singular}\label{s.value ineq obtained from positivity of the block m'x}
    Let \(A,B,X\in \mathcal{M}_n\) be such that \(AX=XA\) and \(M= \begin{bmatrix}
        A & X\\
        X^*& B
    \end{bmatrix}\geq O\). Then \[s_j(X)\leq s_j\Big(A^{1/2}\, B^{1/2}\Big) \; \text{ for } j=1,2,...,n.\]
\end{lemma}
\begin{lemma}\cite{bhatia1990singular}\label{s.value of fractional power and that of the arith mean}
    Let \(A,B\in M_n\) be positive semi-definite. Then \[s_j\Big(A^{1/2}\, B^{1/2}\Big)\leq s_j\Big(\dfrac{A+B}{2}\Big)\; \text{ for } j=1,2,...,n.\]
\end{lemma}

\begin{theorem}
   Let \(A,B,X\in \mathcal{M}_n\) be such that \(M= \begin{bmatrix}
        A & X\\
        X^*& B
    \end{bmatrix}\in \Gamma_n\) and \(AX=XA\). Then \[\Vert X \Vert_u \leq \dfrac{1}{2} \Vert A+B\Vert_u.\]
\end{theorem}
\begin{proof}
    Since \(M\) is accretive, then \(\Re(M)= \begin{bmatrix}
        \Re(A) & X\\
        X^*& \Re(B)
    \end{bmatrix}\geq O\), and since \(AX=XA\), then \(\Re(A)X=X\Re(A)\). Consequently, by Lemma \ref{s.value ineq obtained from positivity of the block m'x}, then Lemma \ref{s.value of fractional power and that of the arith mean}, we obtain
    \begin{align*}
        s_j(X) & \leq s_j\Big( (\Re(A))^{1/2}\, (\Re(B))^{1/2}\Big)\\
        & \leq s_j\Big(\Re\Big(\dfrac{A+B}{2}\Big)\Big)\\
        &\leq s_j\Big(\dfrac{A+B}{2}\Big)\; \text{ for } j=1,2,...,n.\\
         \end{align*}
Equivalently, \(\Vert X \Vert_u \leq \dfrac{1}{2} \Vert A+B\Vert_u\).
\end{proof}

% \bibliography{bib}

\begin{thebibliography}{10}

\bibitem{alakhrass2021note}
M.~Alakhrass.
\newblock A note on positive partial transpose blocks.
\newblock {\em AIMS Math.}, 8(10):23747--23755, 2023.

\bibitem{alakhrass2023singular}
M.~Alakhrass.
\newblock Singular value inequalities related to {PPT} blocks.
\newblock {\em arXiv preprint arXiv:2305.10461}, 2023.

\bibitem{ando2016geometric}
T.~Ando.
\newblock Geometric mean and norm {S}chwarz inequality.
\newblock {\em Ann. Funct. Anal}, 7(1):1--8, 2016.

\bibitem{arlinskii2002sectorial}
Y.~Arlinskii.
\newblock On sectorial block operator matrices.
\newblock {\em Журнал математической физики, анализа, геометрии}, 9(4):533--571, 2002.

\bibitem{ballantine1978numerical}
C.~Ballantine.
\newblock Numerical range of a matrix: some effective criteria.
\newblock {\em Linear Algebra Appl.}, 19(2):117--188, 1978.

\bibitem{bebiano2023numerical}
N.~Bebiano, R.~Lemos, and G.~Soares.
\newblock On the numerical range of kac-sylvester matrices.
\newblock {\em Electron. J. of Linear Algebra}, 39:242--259, 2023.

\bibitem{bedrani2021positive}
Y.~Bedrani, F.~Kittaneh, and M.~Sababheh.
\newblock From positive to accretive matrices.
\newblock {\em Positivity}, 25(4):1601--1629, 2021.

\bibitem{bedrani2021numerical}
Y.~Bedrani, F.~Kittaneh, and M.~Sababheh.
\newblock Numerical radii of accretive matrices.
\newblock {\em Linear Multilinear Algebra}, 69(5):957--970, 2021.

\bibitem{bedrani2021weighted}
Y.~Bedrani, F.~Kittaneh, and M.~Sababheh.
\newblock On the weighted geometric mean of accretive matrices.
\newblock {\em Ann. Func. Anal.}, 12:1--16, 2021.

\bibitem{bhatia2009positive}
R.~Bhatia.
\newblock {\em Positive definite matrices}.
\newblock Princeton university press, 2009.

\bibitem{bhatia2013matrix}
R.~Bhatia.
\newblock {\em Matrix analysis}, volume 169.
\newblock Springer Science \& Business Media, 2013.

\bibitem{bhatia1990singular}
R.~Bhatia and F.~Kittaneh.
\newblock On the singular values of a product of operators.
\newblock {\em SIAM J. Matrix Anal. Appl.}, 11(2):272--277, 1990.

\bibitem{bourin2013positive}
J.C. Bourin, E.Y. Lee, and M.~Lin.
\newblock Positive matrices partitioned into a small number of hermitian blocks.
\newblock {\em Linear Algebra Appl.}, 438(5):2591--2598, 2013.

\bibitem{burqan2017singular}
A.~Burqan and F.~Kittaneh.
\newblock Singular value and norm inequalities associated with 2 x 2 positive semidefinite block matrices.
\newblock {\em Electron. J. Linear Algebra}, 32:116--124, 2017.

\bibitem{chien1998geometric}
M.T. Chien, L.~Yeh, Y.T. Yeh, and F.Z. Lin.
\newblock On geometric properties of the numerical range.
\newblock {\em Linear Algebra Appl.}, 274(1-3):389--410, 1998.

\bibitem{davis1971toeplitz}
C.~Davis.
\newblock The {T}oeplitz-{H}ausdorff theorem explained.
\newblock {\em Can. Math. Bull.}, 14(2):245--246, 1971.

\bibitem{drury2015principal}
S.~Drury.
\newblock Principal powers of matrices with positive definite real part.
\newblock {\em Linear Multilinear Algebra}, 63(2):296--301, 2015.

\bibitem{drury2014singular}
S.~Drury and M.~Lin.
\newblock Singular value inequalities for matrices with numerical ranges in a sector.
\newblock {\em Oper. Matrices}, (4):1143--1148, 2014.

\bibitem{fuglede1950commutativity}
B.~Fuglede.
\newblock A commutativity theorem for normal operators.
\newblock {\em Proc. Natl. Acad. Sci.}, 36(1):35--40, 1950.

\bibitem{furuichi2024further}
S.~Furuichi, H.~R. Moradi, and M.~Sababheh.
\newblock Further properties of accretive matrices.
\newblock {\em Ann. Fenn. Math.}, 49(1):387--404, 2024.

\bibitem{gau2008numerical}
H.~L. Gau and P.~Y. Wu.
\newblock Numerical ranges of nilpotent operators.
\newblock {\em Linear Algebra Appl.}, 429(4):716--726, 2008.

\bibitem{ghaemi2021advances}
M.~B. Ghaemi, N.~Gharakhanlu, T.~M. Rassias, and R.~Saadati.
\newblock {\em Advances in matrix inequalities}.
\newblock Springer, 2021.

\bibitem{gumus2022positive}
I.~Gumus, H.~R. Moradi, and M.~Sababheh.
\newblock On positive and positive partial transpose matrices.
\newblock {\em Electron. J. Linear Algebra}, 38:792--802, 2022.

\bibitem{hiroshima2003majorization}
T.~Hiroshima.
\newblock Majorization criterion for distillability of a bipartite quantum state.
\newblock {\em Phys. Rev. Lett.}, 91(5):057902, 2003.

\bibitem{horodecki1996necessary}
M.~Horodecki, P.~Horodecki, and R.~Horodecki.
\newblock On the necessary and sufficient conditions for separability of mixed quantum states.
\newblock {\em Phys. Lett. A}, 223(1), 1996.

\bibitem{huang2022refining}
H.~Huang.
\newblock Refining some inequalities on 2$\times$2 block accretive matrices.
\newblock {\em Oper. Matrices}, 16(1), 2022.

\bibitem{huang2023singular}
Z.~Huang.
\newblock Singular value inequalities on 2$\times$ 2 block accretive partial transpose matrices.
\newblock {\em ScienceAsia}, 49(6):827--829, 2023.

\bibitem{johnson1972matrices}
C.~R. Johnson.
\newblock {\em Matrices whose {H}ermitian part is positive definite}.
\newblock PhD thesis, California Institute of Technology, 1972.

\bibitem{kato1961fractional}
Tosio Kato.
\newblock Fractional powers of dissipative operators.
\newblock {\em JMSJ}, 13(3):246--274, 1961.

\bibitem{keeler1997numerical}
D.~S. Keeler, L.~Rodman, and I.~M. Spitkovsky.
\newblock The numerical range of 3$\times$ 3 matrices.
\newblock {\em Linear Algebra Appl.}, 252(1-3):115--139, 1997.

\bibitem{kuai2018extension}
L.~Kuai.
\newblock An extension of the fiedler-markham determinant inequality.
\newblock {\em Linear Multilinear Algebra}, 66(3):547--553, 2018.

\bibitem{kubo1980means}
F.~Kubo and T.~Ando.
\newblock Means of positive linear operators.
\newblock {\em Math. Ann.}, 246:205--224, 1980.

\bibitem{lee2015off}
E.Y. Lee.
\newblock The off-diagonal block of a {PPT} matrix.
\newblock {\em Linear Algebra Appl.}, 486:449--453, 2015.

\bibitem{li2021partial}
Y.~Li, X.~Lin, and L.~Feng.
\newblock Partial determinant inequalities for positive semidefinite block matrices.
\newblock {\em J. Math. Inequal.}, 15(4):1435--1445, 2021.

\bibitem{lin2015extension}
M.~Lin.
\newblock Extension of a result of {H}aynsworth and {H}artfiel.
\newblock {\em Arch. Math.}, 104:93--100, 2015.

\bibitem{lin2015inequalities}
M.~Lin.
\newblock Inequalities related to 2$\times$2 block {PPT} matrices.
\newblock {\em Oper. Matrices}, 9(4):917--924, 2015.

\bibitem{lin2016some}
M.~Lin.
\newblock Some inequalities for sector matrices.
\newblock {\em Oper. Matrices}, (4):915--921, 2016.

\bibitem{lin2015hiroshima}
M.~Lin and H.~Wolkowicz.
\newblock Hiroshima’s theorem and matrix norm inequalities.
\newblock {\em Acta Sci. Math.}, 81:45--53, 2015.

\bibitem{liu2021inequalities}
J.~Liu, J.J. Mei, and D.~Zhang.
\newblock Inequalities related to the geometric mean of accretive matrices.
\newblock {\em Oper. Matrices}, 15(2):581--587, 2021.

\bibitem{nasiri2022new}
Leila Nasiri and Shigeru Furuichi.
\newblock New inequalities for sector matrices applying {G}arg--{A}ujla inequalities.
\newblock {\em Advances in Operator Theory}, 7(2):16, 2022.

\bibitem{parker1951characteristic}
W.~Parker.
\newblock Characteristic roots and field of values of a matrix.
\newblock {\em Bull. Amer. Math. Soc.}, 57(2):103--108, 1951.

\bibitem{peres1996separability}
A.~Peres.
\newblock Separability criterion for density matrices.
\newblock {\em Phys. Rev. Lett.}, 77(8):1413, 1996.

\bibitem{pusz1975functional}
W.~Pusz and S.~Woronowicz.
\newblock Functional calculus for sesquilinear forms and the purification map.
\newblock {\em Rep. Math. Phys.}, 8(2):159--170, 1975.

\bibitem{putnam1951normal}
C.~Putnam.
\newblock On normal operators in {H}ilbert space.
\newblock {\em Am. J. Math.}, 73(2):357--362, 1951.

\bibitem{raissouli2017relative}
M.~Ra{\"\i}ssouli, M.~Sal Moslehian, and S.~Furuichi.
\newblock Relative entropy and {T}sallis entropy of two accretive operators.
\newblock {\em C. R. Math.}, 355(6):687--693, 2017.

\bibitem{sababheh2023matrix}
M.~Sababheh, C.~Conde, and H.R. Moradi.
\newblock On the matrix {C}auchy-{S}chwarz inequality.
\newblock {\em Oper. Matrices}, 17(2), 2023.

\bibitem{sababheh2024further}
M.~Sababheh, I.~Gumus, and H.~R. Moradi.
\newblock Further propertiers of {PPT} and $(\alpha,\beta)$-normal matrices.
\newblock {\em Oper. Matrices}, 18(1):257--271, 2024.

\bibitem{sababheh2022operator}
M.~Sababheh, I.~G{\"u}m{\"u}{\c{s}}, and H.~R. Moradi.
\newblock Operator inequalities via accretive transforms.
\newblock {\em Hacet. J. Math. Stat.}, 53(1):40--52, 2024.

\bibitem{sababheh2023new}
M.~Sababheh and H.R. Moradi.
\newblock New orders among {H}ilbert space operators.
\newblock {\em Math. Inequal. Appl.}, 26(2), 2023.

\bibitem{southworth2020note}
B.~Southworth and S.~Olivier.
\newblock A note on 2$\times$2 block-diagonal preconditioning.
\newblock {\em arXiv preprint arXiv:2001.00711}, 2020.

\bibitem{tan2020extension}
F.~Tan and A.~Xie.
\newblock An extension of the {AM--GM--HM} inequality.
\newblock {\em Bull. Iran. Math. Soc.}, 46:245--251, 2020.

\bibitem{tao2006more}
Y.~Tao.
\newblock More results on singular value inequalities of matrices.
\newblock {\em Linear Algebra Appl.}, 416(2-3):724--729, 2006.

\bibitem{yang2022some}
Chaojun Yang.
\newblock Some singular value inequalities for sector matrices involving operator concave functions.
\newblock {\em J. Math.}, 2022(1):4535343, 2022.

\bibitem{yang2022inequalities}
J.~Yang.
\newblock Inequalities on 2$\times$ 2 block accretive matrices.
\newblock {\em Oper. Matrices}, 16:323--328, 2022.

\bibitem{f2008numerical}
P.~Zachlin and M.~Hochstenbach.
\newblock On the numerical range of a matrix: By {R}udolf {K}ippenhahn (1951 in {B}omberg).
\newblock {\em Linear Multilinear Algebra}, 56(1-2):185--225, 2008.

\bibitem{zhang2015matrix}
F.~Zhang.
\newblock A matrix decomposition and its applications.
\newblock {\em Linear Multilinear Algebra}, 63(10):2033--2042, 2015.

\end{thebibliography}
\bibliographystyle{plain}

\end{document}